\documentclass[11pt,reqno]{amsart}
\usepackage{amsmath,amsthm,amscd,amsfonts,amssymb,color}
\usepackage{cite}
\usepackage[mathscr]{eucal}

\usepackage[bookmarksnumbered,colorlinks,plainpages]{hyperref}
\newtheorem{theorem}{Theorem}[section]
\newtheorem{lemma}[theorem]{Lemma}
\newtheorem*{Acknowledgement}{\textnormal{\textbf{Acknowledgement}}}

\newtheorem{proposition}[theorem]{Proposition}
\newtheorem{corollary}[theorem]{Corollary}

\theoremstyle{definition}
\newtheorem{definition}[theorem]{Definition}
\newtheorem{example}[theorem]{Example}
\newtheorem{Open Prob}[theorem]{Open Problem}
\newtheorem{remark}[theorem]{Remark}
\numberwithin{equation}{section}
\def\DJ{\leavevmode\setbox0=\hbox{D}\kern0pt\rlap{\kern.04em\raise.188\ht0\hbox{-}}D}
\begin{document}
	
	\title[M Compactness in Normed Linear Spaces]{Characterization of $M$-compact sets via statistically convergent sequences}
	
	\author[S.\ Seal, S.\ Som, S.\ Basu, L.K.\ Dey]
	{Susmita Seal$ ^{1}$, Sumit Som$^{2}$, Sudeshna Basu$^{3}$, Lakshmi Kanta Dey$^{4}$}
	
	\address {{$^{1}$} Susmita Seal, 
		Department of Mathematics,
		Ram Krishna Mission Vivekananda Education and Research Institute, 
		Belur Math,  Howrah 711202,
		West Bengal, India}
	\email{susmitaseal1996@gmail.com}
	
    \address{{$^{2}$} Sumit Som,
		Department of Mathematics,
		National Institute of Technology
		Durgapur, India.}
	\email{somkakdwip@gmail.com}
	
	\address{{$^{3}$}   Sudeshna Basu,
		Department of Mathematics, 
		Ram Krishna Mission Vivekananda Education and Research Institute, 
		Belur Math,  Howrah 711202
		West Bengal, India and 
		Department of Mathematics,
		George Washington University,
		Washington DC 20052 USA
		}
	\email{sbasu@gwu.edu, sudeshnamelody@gmail.com}
	\address{{$^{4}$} Lakshmi Kanta Dey,
		Department of Mathematics,
		National Institute of Technology
		Durgapur, India.}
	\email{lakshmikdey@yahoo.co.in}
	
	\keywords{ Statistical convergence, ideal, maximizing sequence, $M$-compact.\\
	\indent 2010 {\it Mathematics Subject Classification}.  46B20; 41A65; 41A50; 40A35.}
	
	\date{}
	
	\setcounter {page}{1}
	
	

\begin{abstract}
In this paper, we study stability of $M$-compactness for $l^p$ sum of Banach spaces for $1 \leq p<\infty$. We also obtain a characterization of $M$-compact sets in terms of statistically maximizing sequence, a notion which is weaker than a maximizing sequence. Moreover, we introduce the notion of  $\mathcal{I}$-$M$-compactness of a bounded subset $M$ of a normed linear space $X$ with respect to an ideal $\mathcal{I}$ and show that it is equivalent to $M$-compactness for non-trivial admissible ideals.
\end{abstract}
	
\maketitle 
	
\section{Introduction}
Let $X$ be a real normed linear space and $G$ be a nonempty, bounded subset of $X$. For any $x\in X$, the farthest distance from $x$ to the set $G$ is denoted by $\delta(x,G)$ and is defined by $$\delta(x,G)=~\mbox{sup}\Big\{\Vert x-e \Vert : e\in G\Big\}.$$  Throughout our discussion, we consider only bounded subsets of $X$ only.
We recall the following two definitions: 

\begin{definition}\cite{sain}
Let $X$ be a real normed linear space and $M$ be a nonempty bounded subset of $X$. A sequence $\{x_{n}\}_{n\in \mathbb{N}}$ in $X$ is said to be maximizing if there exists $x\in X$ such that  $\Vert x_{n} - x \Vert \rightarrow \delta(x,M)$ as $n\rightarrow \infty.$
\end{definition}

\begin{definition}\cite{sain}
Let $S\subseteq X$. Then $S$ is said to be $M$-compact if every maximizing sequence in $M$ has a convergent subsequence.
\end{definition}

In this paper we study certain stability results for $M$-compact sets in Banach spaces. $M$-compact sets were first introduced and subsequently studied by Vlasov \cite{v}. It is not difficult to show that $M$-compactness is a proper generalization of the usual compactness and $M$-compact sets may not be closed. $M$-compact sets have been widely studied in the context of farthest distance maps \cite{sain}. Recently, it was proved that an uniquely remotal $M$-compact set with compact derived set is a singleton\cite{sain}. 

In this work, we prove the stability of $M$-compact sets with respect to $\ell_{p}$ sum for $ 1 \leq p <\infty$. We also give a characterization of $M$-compact sets in terms of statistically maximizing sequence in a normed linear space introduced in \cite{sdb}. We prove the stability of $M$-compact sets  under $\ell_{p}$ sum. We introduce a new notion of convergence of sequences in a normed linear space, namely, $\mathcal{I}$-convergence and use it to define $\mathcal{I}$-maximizing sequence and $\mathcal{I}$-$M$ compactness and explore its relationship to $M$-compactness. Statistical convergence for real numbers  was introduced in \cite{fas} and subsequently studied in great details in \cite{ma,s}. Statistical limit points were studied extensively in \cite{con1, con2, fri1, fri2, kos}. It was also studied in context of measure theory in \cite{mil}. In \cite{dum}, statistical convergence was studied in the context of approximation theory and in \cite{con3}, it was studied in the context of Banach Spaces. 
The notion of $\mathcal{I}$-convergence for real numbers was extensively studied in \cite{pav, pav1} in an attempt to generalize the notion statistical convergence for real numbers.

\section{Stability of $M$-compactness under $\ell_p$-sum } 

\begin{lemma}\label{lpsum}
Let $X$ and $Y$ be two normed linear spaces and $Z$ = $X$ $\oplus_{p}$ $Y$($1 \leqslant p < \infty$). Let $ M\subseteq X$ and $N\subseteq Y$ be bounded subsets of $X$ and $Y$ respectively. Then $z_{n}=  (x_{n},y_{n}) \in M\oplus_{p}N$ is a maximizing sequence in $M\oplus_{p}N$  if and only if $(x_{n})$ and $(y_{n})$ are maximizing sequences in $M$ and $N$ respectively. 
\end{lemma}
\begin{proof}
Firstly, let $\{z_{n}\}_{n\in \mathbb{N}}$ be a maximizing sequence in $M\oplus_{p}N$.
Then there exists $ z\in X\oplus_{p} Y$ such that as $ n\rightarrow \infty,$
\begin{eqnarray*}
	\| z_n-z \|_{p} & \rightarrow & \delta(z,M\oplus_{p}N); \\
	\Longrightarrow \| (x_n,y_n)-(x,y) \|_{p} & \rightarrow &\delta((x,y),M\oplus_{p}N)=  (\delta(x,M)^p +\delta(y,N)^p)^{1/p}; \\
	\Longrightarrow (\| x_n-x \| ^p + \| y_n-y\| ^p)^{1/p} &\rightarrow & (\delta(x,M) ^p+\delta(y,N) ^p)^{1/p}; \\
	\Longrightarrow \| x_n-x \| ^p + \| y_n-y\| ^p &\rightarrow & \delta(x,M) ^p+\delta(y,N) ^p); \\
	\Longrightarrow (\delta(x,M)^p-\| x_n-x \| ^p)&+&(\delta(y,N)^p-\| y_n-y \| ^p)  \rightarrow 0. \\
	\mbox{Since,} \quad  \delta(x,M)^p-\| x_n-x \| ^p \geqslant 0 & \mbox{and} & \delta(y,N)^p-\Vert y_n-y \Vert ^p \geqslant 0 \quad  \forall~ n, \\
	\mbox{we have,}~  \delta(x,M)^p-\| x_n-x \|^p \rightarrow 0 & and & \delta(y,N)^p-\| y_n-y \| ^p \rightarrow 0. 
\end{eqnarray*}

Claim: $\| x_n - x \| \rightarrow \delta(x,M)$ as $n\rightarrow \infty.$

If $\delta(x,M)$ = 0 then  $\| x_n - x \| = 0 = \delta(x,M)~ \mbox{for all } n \in \mathbb{N}.$ 

If $\delta(x,M) \neq 0,$ choose $0< b < \delta(x,M).$ Then there exists  $N_0 \in \mathbb{N}$ such that $\delta(x,M) ^p \geq \| x_n - x \| ^p > b^p$  $\forall~ n\geq N_0 $ i.e. $\delta(x,M)  \geq \| x_n - x \|  > b$  $  \forall ~n\geq N_0.$

Thus, $\forall n \geq N_0,$ we have,
\begin{eqnarray*}
	0 & \leq & \delta(x,M) - \| x_n - x \| \\
	& = &  \frac{ \delta(x,M)^p - \| x_n - x \| ^p }{\delta(x,M)^{p-1} + \delta(x,M)^{p-2} \| x_n - x \| +\dots+ \| x_n - x \| ^{p-1}  }\\
	& \leq & \frac{\delta(x,M)^p - \|x_n - x \| ^p}{b^{p-1} p}\rightarrow  0 \quad \mbox{as}\quad  n\rightarrow \infty.
\end{eqnarray*}
Hence $\| x_n - x \| \rightarrow \delta(x,M).$ Similarly, $\Vert y_n - y \Vert \rightarrow \delta(y,N)~\mbox{as}~n\rightarrow \infty.$

Conversely, let $\{x_{n}\}_{n\in \mathbb{N}}$ and $\{y_{n}\}_{n\in \mathbb{N}}$ be  maximizing sequences in $M$ and $N$ respectively. Then there exist $x\in X$ and $y \in Y$  such that  $\| x_n-x \|  \rightarrow  \delta(x,M)$ and  $\| y_n-y \|  \rightarrow  \delta(y,N)$ as $n\rightarrow \infty.$ Thus 
$$\| x_n-x \| ^p + \| y_n-y \| ^p \rightarrow  \delta(x,M)^p + \delta(y,N)^p~ \mbox{as}~ n\rightarrow \infty,$$
i.e. $\| (x_n,y_n)-(x,y) \|_{p} ^p \rightarrow   \delta(x,M)^p+\delta(y,N)^p  =\delta((x,y),M\oplus_{p}N)^p~ \mbox{as}~ n\rightarrow \infty.$
Hence, 
$$\| (x_n,y_n)-(x,y) \|_{p}  \rightarrow \delta((x,y),M \oplus_{p} N)~ \mbox{as}~ n\rightarrow \infty.$$
\end{proof}
 
\begin{remark} \label{n max}
Arguing similarly, we have, suppose $X = \oplus_{l^p} X_i$ and $M_i \subset X_i$ for all $i\in \mathbb{N}.$ Let 
 $\{x_n\}_{n\in \mathbb{N}}$ be a maximizing sequence in $\oplus_{l^p} M_i.$ Then the $i$-th co-ordinate sequence $\{x^{(i)}_n\}_{n\in \mathbb{N}}$ is maximizing in $M_i$ for all $i\in \mathbb{N}$ ($1\leqslant p < \infty$). 
\end{remark}

\begin{proposition}\label{aa}
Let $X$ and $Y$ be two normed linear spaces and $Z$ = $X$ $\oplus_{p}$ $Y$. Then $H\subset X$ and $N\subset Y$ are $M$-compact if and only if $H$ $\oplus_{p}$ $N$ is $M$-compact where $1 \leqslant p < \infty.$
\end{proposition}

\begin{proof}
Firstly, let $H$ and $N$ be $M$-compact subsets in $X$ and $Y$ respectively. Let $\{z_{n}\}_{n\in \mathbb{N}}$ be a maximizing sequence in $H$ $\oplus_{p}$ $N$. Let  $z_{n}=(x_{n},y_{n}),~ x_{n} \in H,~ y_{n}\in N~\mbox{for all}~n\in \mathbb{N}.$ Then by Lemma \ref{lpsum}, $\{x_{n}\}_{n\in \mathbb{N}}$ and $\{y_{n}\}_{n\in \mathbb{N}}$ are maximizing sequences in $H$ and  $N$ respectively.
Since $H$ is $M$-compact, $\{x_{n}\}_{n\in \mathbb{N}}$ has a convergent subsequence $\{ x_{n_k} \}_{k\in \mathbb{N}}.$ Since, $\| y_{n_k}-y \| \rightarrow \delta(y,N)~\mbox{as}~k\rightarrow \infty,$ so, $\{y_{n_k}\}_{k\in \mathbb{N}}$ is a maximizing sequence in $N$.
Since $N$ is $M$-compact, so, $\{y_{n_k}\}_{k\in \mathbb{N}}$ has a convergent subsequence $\{y_{n_{k_l}}\}_{l\in \mathbb{N}}.$ It follows that, $\{(x_{n_{k_l}},y_{n_{k_l}})\}_{l\in \mathbb{N}}$ is a convergent subsequence of $\{(x_{n},y_{n})\}_{n\in \mathbb{N}}.$ Thus $H \oplus_{p} N $  is $M$-compact.
 
Conversely, let $H \oplus_{p} N$ be $M$-compact. So, $H \oplus_{p} N$ is bounded and consequently, $H$ and $N$ both are bounded subsets of $X$ and $Y$ respectively. Let $\{x_{n}\}_{n\in \mathbb{N}}$ be a maximizing sequence in $H$. Then there exists $ x \in X$ such that $\| x_n-x \|  \rightarrow  \delta(x,H)$ as $n\rightarrow \infty.$ 
Since, $N$ is bounded, so, for fix  $y \in Y,~ \delta(y,N)$ exists. So, there exists a sequence $\{y_{n}\}_{n\in \mathbb{N}} \in N$ such that  $\| y_n-y \|  \rightarrow  \delta(y,N)$ as $n\rightarrow \infty.$ 
Hence by Lemma \ref{lpsum},~$z_{n}=(x_{n} , y_{n}) $ is a maximizing sequence in $H\oplus_{p} N$ and   $\| (x_n,y_n)-(x,y) \|_{p}  \rightarrow \delta\Big((x,y),H \oplus_{p} N\Big)$ as $n\rightarrow \infty.$
Since $H\oplus_{p} $ N is $M$-compact, $\{(x_{n},y_{n})\}_{n\in \mathbb{N}}$ has a convergent subsequence $\{(x_{n_k} y_{n_k})\}_{k\in \mathbb{N}}.$ It follows that, $\{x_{n_k}\}_{k\in \mathbb{N}}$ is a convergent subsequence of  $\{x_{n}\}_{n\in \mathbb{N}}$.
Thus $H$ is $M$-compact. Similarly, $N$ is $M$-compact.
\end{proof}

Our next example shows that Proposition $\ref{aa}$ is not true for p = $\infty$. 

\begin{example}
In $\mathbb{R},$ let $A = [-5,5] , B= [-1,1] \setminus  \{0 \}.$ Then $A$ and $B$ are both $M$-compact. But $A \oplus _{\infty} B$ is not $M$-compact. Indeed, let $\{(5- \frac{1}{n},\frac{1}{n})\}_{n\in \mathbb{N}}$ be  a sequence in $A \oplus _{\infty} B.$ Then 
$\| (5- \frac{1}{n},\frac{1}{n}) -0 \|  \rightarrow  \delta\Big(0,A \oplus _{\infty} B\Big)$ as $n\rightarrow \infty$ and hence, a maximizing sequence. We use the fact $\delta\Big((x,y),M\oplus_{\infty}N\Big)= \mbox{max}\Big \{ \delta(x,M) , \delta(y,N)\Big \}.$ But, $\{(5- \frac{1}{n},\frac{1}{n})\}_{n\in \mathbb{N}}$ has no convergent subsequence in $A \oplus _{\infty} B$. Consequently, $A \oplus _{\infty} B$ is not $M$-compact.
\end{example}

\section{\textbf{ $M$-compact sets and Statistically Convergent Sequences  }}


\begin{definition}\cite{s}
A sequence $\{x_{n}\}_{n\in \mathbb{N}}$ in a normed linear space $X$ is said to be statistically convergent to $x\in X$ if for each $\varepsilon>0,$ 
$$\displaystyle{\lim_{n\rightarrow \infty}}\frac{1}{n} \Big |\Big\{k\leq n : \Vert x_{k}-x \Vert \geq \varepsilon\Big\} \Big |=0.$$ 
If a sequence $\{x_{n}\}_{n\in \mathbb{N}}$ statistically converges to $x,$ then we denote it as, stat-lim $x_n=x.$
\end{definition}

We recall the following definition from \cite{sdb}.

\begin{definition}\cite{sdb}
Let $X$ be a real normed linear space and $M$ be a nonempty, bounded subset of X. A sequence $\{x_n\}_{n\in \mathbb{N}}\subset M$ is said to be statistically maximizing if there exists $x\in X$ such that $\Big\{\Vert x_n-x \Vert\Big\}_{n \in \mathbb{N}}$ is statistically convergent to $\delta(x,M)=\mbox{sup}\Big\{\Vert x-y \Vert : y\in M\Big\}$ as $n\rightarrow \infty.$
\end{definition}

It was proved in \cite {sdb} that a maximizing sequence is statistically maximizing. But the converse is not true. We quote the example from \cite{sdb} for completeness.
\begin{example}
	Let $\mathbb{R}$ be the set of all real numbers. Let $M= [-1,1].$ Now 
	$$\delta(0,M)=\mbox{sup}\Big\{\vert x \vert : x\in M\Big\}=1.$$
	Let  $\{x_{n}\}_{n\in \mathbb{N}}$ in $[-1,1]$  be defined as
	\begin{equation*}
	~x_{n}=
	\begin{cases}
	0  ~\mbox{if} ~~ n=m^{2}~\mbox{for some}~ m\in \mathbb{N},  ~&\\
	1-\frac{1}{n} ~\mbox{if} ~n\neq m^{2}~\mbox{for all}~ m\in \mathbb{N}.
	\end{cases}
	\end{equation*}
	
Then $\{x_{n}\}_{n\in \mathbb{N}}$ is not maximizing. Indeed, for any $x\in \mathbb{R},$ the real sequence $\Big\{|x_{n}-x|\Big\}_{n\in \mathbb{N}}$ is not convergent in $\mathbb{R}.$ But, $\{x_{n}\}_{n\in \mathbb{N}}$ is statistically convergent to $\delta(0,M)=1.$ Indeed, let $0<\varepsilon \leq 1.$ Now
	
	\begin{align}
	\frac{1}{n} \Big\vert\{k\leq n : \vert x_{k}-1 \vert \geq \varepsilon\}\Big\vert &= \frac{1}{n} \Big([\sqrt{n}]+d\Big) \nonumber\\
	&\leq \frac{\sqrt{n}}{n}+ \frac{d}{n} \nonumber
	\end{align}
	
	$$\Longrightarrow \displaystyle{\lim_{n\rightarrow \infty}}\frac{1}{n}\Big\vert\{k\leq n : \vert x_{k}-1 \vert \geq \varepsilon\}\Big\vert = 0.$$
Here $d$ is a finite positive integer. Hence, the sequence $\{x_{n}\}_{n\in \mathbb{N}}$ is statistically maximizing in $M.$ 
\end{example}

We need the following lemma and omit the easy proof.

\begin{lemma} \label{stat sum product}
Let $X$ and $Y$ be two normed linear spaces.
\begin{enumerate}
\item If $\{x_{n}\}_{n\in \mathbb{N}}$ and $\{y_{n}\}_{n\in \mathbb{N}}$ statistically converge to x and y respectively then $\{x_{n} + y_{n}\}_{n\in \mathbb{N}}$ statistically converges to x+y.

\item Let $\{x_{n}\}_{n\in \mathbb{N}}$ be a real sequence such that $0\leqslant x_n \leqslant x$ for all $n.$ Then $\{x_{n}^p\}_{n\in \mathbb{N}}$ statistically converges to $x^p$ if and only if $\{x_{n}\}_{n\in \mathbb{N}}$ statistically converges to $x$ for $1\leqslant p < \infty.$
\end{enumerate}
\end{lemma}

\begin{proposition}\label{stat lpsum}
Let $X$ and $Y$ be two normed linear spaces and $Z = X \oplus_{p} Y (1 \leqslant p < \infty$). Let $ M\subseteq X$ and $N\subseteq Y$ be bounded subsets of $X$ and $Y$ respectively. Then $z_{n}=  (x_{n},y_{n}) \in M\oplus_{p}N$ is a statistically maximizing sequence in $M\oplus_{p}N$  if and only if $(x_{n})$ and $(y_{n})$ are statistically maximizing sequences in $M$ and $N$ respectively. 
\end{proposition}

\begin{proof}
Firstly, let $\{x_{n}\}_{n\in \mathbb{N}}$ and $\{y_{n}\}_{n\in \mathbb{N}}$ be statistically maximizing in $M$ and $N$ respectively. Then there exists $x_0 \in X$ and $y_0 \in Y$ such that stat-lim $\Vert x_n - x_0 \Vert$ = $\delta(x_0,M)$ and stat-lim $\Vert y_n - y_0 \Vert$ = $\delta(y_0 , N).$ Now, from Lemma \ref{stat sum product}, stat-lim $\Vert x_n - x_0 \Vert ^p$ + $\Vert y_n - y_0 \Vert^p$ = $\delta(x_0 , M)^p + \delta(y_0 , N)^p.$ 

$$\Longrightarrow~ \mbox{stat-lim}~ \Vert (x_n,y_n) - (x_0,y_0) \Vert_p ^p = \delta\Big((x_0 , y_0) , M \oplus _p N\Big)^p;$$  

$$\Longrightarrow~ \mbox{stat-lim}~\Vert (x_n,y_n) -(x_0,y_0) \Vert_p=\delta\Big((x_0 , y_0) , M \oplus _p N\Big).$$

Conversely, let $\{(x_{n},y_n)\}_{n\in \mathbb{N}}$ be statistically maximizing in $M \oplus_p N$. Then there exists $(x_0 ,y_0) \in X \oplus_p Y $ such that stat-lim $\Vert (x_n,y_n) - (x_0,y_0) \Vert_p $ = $\delta\Big((x_0 , y_0),M \oplus _p N\Big).$ Then by Lemma \ref{stat sum product},

$$\mbox{stat-lim}~ \Vert x_n - x_0 \Vert ^p + \Vert y_n - y_0 \Vert^{p} = \delta(x_0 , M)^p + \delta(y_0 , N)^p.$$ 

Let $\varepsilon > 0.$ Then 
$$\displaystyle{\lim_{n\rightarrow \infty}}\frac{1}{n} \Big |\Big\{k\leq n : \delta(x_0, M)^p -\Vert x_k - x_0 \Vert^p \geq \varepsilon\Big\} \Big |$$ 

$$\leqslant \displaystyle{\lim_{n\rightarrow \infty}}\frac{1}{n} \Big |\Big\{k\leq n : \delta(x_0 , M)^p -\Vert x_k - x_0 \Vert^p + \delta(y_0 , N)^p -\Vert y_k - y_0 \Vert^p \geq \varepsilon\Big\} \Big |$$  

$$=\displaystyle{\lim_{n\rightarrow \infty}}\frac{1}{n} \Big |\Big\{k\leq n : \delta(x_0 , M)^p + \delta(y_0 , N)^p-(\Vert x_k - x_0 \Vert^p  +\Vert y_k - y_0 \Vert^p) \geq \varepsilon\Big\} \Big |= 0.$$ So, stat-lim $\Vert x_n - x_0 \Vert^p = \delta(y_0,N)^p.$ Similarly, stat-lim $\Vert y_n - y_0 \Vert^p$ = $\delta(y_0 , N)^p$ . By Lemma \ref{stat sum product}, we have, stat-lim $\Vert x_n - x_0 \Vert$ = $\delta(x_0 , M)$ and stat-lim $\Vert y_n - y_0 \Vert$ = $\delta(y_0,N).$
\end{proof} 

\begin{corollary} \label{stat max}
Suppose $X=\oplus_{l^p} X_i$ and $M_i \subseteq X_i.$ Let $\{x_{n}\}_{n\in \mathbb{N}}$ be a statistically maximizing sequence in $  \oplus_{l^p} M_i.$ Then $\{x^{(i)}_n\}_{n\in \mathbb{N}}$ is statistically maximizing in $M_i$ for all $i\in \mathbb{N}$ for $1\leqslant p < \infty.$ 
\end{corollary}

\begin{remark}
We do not know whether the converse of Corollary \ref {stat max} is true or not.
\end{remark}

The following lemma is easy to prove.
\begin{lemma} \label {stat convgt lemma}(\cite {s})
$\{x_{n}\}_{n\in \mathbb{N}}$ is statistically convergent to $x$ if and only if there exists a set K= $\{ k_1<k_2<k_3<\dots\} \subseteq \mathbb{N} $ s.t. $\displaystyle{\lim_{n\rightarrow \infty}}$ $x_{k_n}$ = x and $ \displaystyle{\lim_{n\rightarrow \infty}} \frac{1}{n}|K \cap [1,n]|=1.$ 
\end{lemma}

\begin{proposition} \label{charecterisation}
	Suppose $A\subseteq X$. Then the the following are equivalent:
\begin{enumerate}
\item $A$ is $M$-compact.
\item Every statistically maximizing  sequence in $A$ has a convergent subsequence.		
\item Every statistically maximizing sequence in $A$ has a statistically convergent subsequence.
\end{enumerate} 
\end{proposition}

\begin{proof}
	$(1)\Rightarrow (2)$
Let $\{x_{n}\}_{n\in \mathbb{N}}$ be a  statistically maximizing sequence in A. Then there exists $x \in X$ such that  $\|{x_{n} - x\|}_{n\in \mathbb{N}}$ is statistically convergent to $\delta(x,A).$ By Lemma \ref{stat convgt lemma}, $\{ \| x_{n} - x \| \}_{n\in \mathbb{N}}$ has a subsequence $\{ \| x_{n_k} - x \| \}_{k\in \mathbb{N}}$ such that  $\| x_{n_k}- x \| \rightarrow$ $\delta(x,A)$ as $k\rightarrow \infty.$ Thus $\{x_{n_k}\}_{k\in \mathbb{N}}$ is a maximizing sequence in A. Since $A$ is $M$-compact, so, $\{x_{n_k}\}_{k\in \mathbb{N}}$ has a convergent subsequence. Consequently,  $\{x_{n}\}_{n\in \mathbb{N}}$ also has a convergent subsequence.

$(2) \Rightarrow(3)$ Obvious.
	
$(3)\Rightarrow (1)$ Let $\{x_{n}\}_{n\in \mathbb{N}}$ be a maximizing sequence in $A.$ Then there exists $ x \in X$ s.t. $\| x_n - x \|\rightarrow$ $\delta(x,A)$ as $n\rightarrow \infty.$ Hence  $\{\|x_{n} - x\|\}_{n\in \mathbb{N}}$ is statistically convergent to $\delta(x,A),$ i.e., $\{x_{n}\}_{n\in \mathbb{N}}$ is a statistically maximizing sequence in $A.$ Hence it has a statistically convergent subsequence $\{x_{n_k}\}$ in $A.$ By Lemma \ref{stat convgt lemma}, there exists a convergent subsequence $\{x_{n_{k_l}}\}$ of $\{x_{n_k}\}.$
Consequently, $\{x_{n_{k_l}}\}$ is a convergent subsequence of $\{x_n\}.$ Thus A is $M$-compact.
\end{proof}

\begin{remark} 
\begin{enumerate}
\item Proposition \ref{charecterisation} gives a characterization of $M$-compact sets in terms of statistically maximizing sequence which is a weaker notion than maximizing sequence. 
\item It was proved in \cite{sdb} that if  $M \subseteq X.$ If $\{x_n\}_{n\in \mathbb{N}}$ is a statistically maximizing sequence in $M$ then $\{x_n\}_{n\in \mathbb{N}}$ is a statistically maximizing sequence in $\overline{M}.$
\end{enumerate}
\end{remark}

\section{$\mathcal{I}$-Convergence and $\mathcal{I}$-$M$-compactness }

We first recall the  following definitions from \cite {pav}.

\begin{definition}\cite{pav}
Let $X\neq \phi$ . A class $\mathcal{I}$ $\subset$ $2^{X}$  of subsets of X is said to be an ideal in X provided that  $\mathcal{I}$ is additive and hereditary, i.e.,  $\mathcal{I}$  satisfies the following conditions:
\begin{enumerate}
\item$\phi \in \mathcal{I}.$ 
\item$A,B \in \mathcal{I}$ $\Longrightarrow$ $ A \cup B \in  \mathcal{I}$  (additivity).
\item $A \in \mathcal{I}$ , $B \subset A$   $\Longrightarrow$ $ B \in \mathcal{I}$ (hereditary).
\end{enumerate}
An ideal is called non-trivial if $X$ $\notin \mathcal{I}.$ 
\end{definition}

\begin{definition}\cite{pav}
A non-trivial ideal $\mathcal{I}$ in $X$ is called admissible if $\{x\} \in \mathcal{I}$ for each $x \in X.$
\end{definition}

\begin{definition}\cite{pav}
Let $\mathcal{I}$ be a non trivial ideal in $\mathbb{N}$. A sequence $\{x_{n}\}_{n\in \mathbb{N}}$ of real numbers is said to be $\mathcal{I}-$convergent to $x\in \mathbb{R}$ if for every $\varepsilon>0,$ the set $A(\varepsilon) =\{n$ : $\vert x_n - x \vert \geq \varepsilon \} \in \mathcal{I}.$ 
\end{definition}

Likewise, we define,

\begin{definition}
Let $\mathcal{I}$ be a non trivial ideal in $\mathbb{N}$ and $X$ be a real normed linear space. A sequence $\{x_{n}\}_{n\in \mathbb{N}}$ in $X$ is said to be $\mathcal{I}$-convergent to $x\in X$ if for every $\varepsilon>0,$ the set $A(\varepsilon)=\{n : \Vert x_n - x \Vert \geq \varepsilon \}\in \mathcal{I}.$ If a sequence $\{x_{n}\}_{n\in \mathbb{N}}$ $\mathcal{I}$-converges to $x,$ then we denote it as, $\mathcal{I}$-lim $x_n=x.$
\end{definition}

\begin{remark} \label{ag}
Clearly if $\mathcal{I}$ is a non-trivial admissible ideal in $\mathbb{N}$ , then norm convergence implies $\mathcal{I}$-convergence.
\end{remark}

\begin{proposition} \label{ah}
If $\mathcal{I}$ is a non-trivial admissible ideal in $\mathbb{N}$ then $\{x_n\}_{n\in \mathbb{N}}$ $\mathcal{I}$-converges to $x$ implies $\{x_n\}_{n\in \mathbb{N}}$ has a convergent subsequence which converges to $x.$
\end{proposition}

\begin{proof}
If possible, let $\{x_{n}\}_{n\in \mathbb{N}}$ has no such convergent subsequences. Then there exists $\varepsilon>0$ and there exists $N_0 \in \mathbb{N}$ such that
$$\Vert x_{n} -x \Vert \geq \varepsilon~ \forall ~ n\geq N_0.$$
Since $\{x_{n}\}_{n\in \mathbb{N}}$  is $\mathcal{I}$-convergent to $x,$ so, $\{n\in \mathbb{N} : \Vert x_n -x \Vert \geq \varepsilon\} \in \mathcal{I}.$ Also, $\{ N_0, N_0+1, N_0+2,\dots\} \subseteq \{n\in \mathbb{N}$ : $\Vert x_n -x \Vert \geq \varepsilon\}\Longrightarrow \{N_0, N_0+1, N_0+2,\dots\} \in \mathcal{I}.$ Now, since $\mathcal{I}$ is admissible, so $\{1,2,\dots, N_0-1\} \in \mathcal{I}.$ Hence, $\mathbb{N}= \{1,2,\dots,N_0-1\} \bigcup \{N_0,N_0+1,N_0+2,\dots\} \in \mathcal{I},$ which is a contradiction (since $\mathcal{I}$ is non trivial).
\end{proof}

Now, we recall  an important example of $\mathcal{I}$ convergence from \cite{pav}.
\begin{example}
Let $\emptyset \neq M \subsetneqq \mathbb{N}.$ Put $\mathcal{I}_M = 2^M.$ Then $\mathcal{I}_M$ is a non trivial ideal in $\mathbb{N}$ which is not admissible. Also, a sequence $\{x_{n}\}_{n\in \mathbb{N}}$ is  $\mathcal{I}_M$-convergent if and only if it is constant in $\mathbb{N} \setminus M.$ 
\end{example}

\begin{definition}
Let $X$ be a real normed linear space and $M$ be a nonempty, bounded subset of $X$. A sequence $\{x_n\}_{n\in \mathbb{N}}\subset M$ is said to be $\mathcal{I}$-maximizing if there exists $x\in X$ such that $\Big\{\| x_n-x \|\Big\}_{n \in \mathbb{N}}$ is $\mathcal{I}$-convergent to $\delta(x,M)=\mbox{sup}\Big\{\Vert x-y \Vert : y\in M\Big\}.$
\end{definition}
We have the following lemma. We omit the easy proof.

\begin{lemma} \label{I convergence sum product}
\begin{enumerate}

\item \cite{pav} If the real sequences $\{x_{n}\}_{n\in \mathbb{N}}$ and $\{y_{n}\}_{n\in \mathbb{N}}$ are $\mathcal{I}$-convergent to $x$ and $y$ respectively  then $\{x_{n} + y_{n}\}_{n\in \mathbb{N}}$ is $\mathcal{I}$-convergent to $x+y.$
\item If $\{x_{n}\}_{n\in \mathbb{N}}$ be a real sequence such that $0\leq x_n \leq x$ for all $n.$ Then $\{x_{n}^p\}_{n\in \mathbb{N}}$ $\mathcal{I}$ converges to $x^p$ if and only if $\{x_{n}\}_{n\in \mathbb{N}}$ $\mathcal{I}$ converges to x. ($1\leq p<\infty$)
\end{enumerate}
\end{lemma}

\begin{proposition}\label{I lpsum}
Let $X$ and $Y$ be two normed linear spaces and $Z = X \oplus_{p} Y$ ($1 \leq p <\infty$). Let $ M\subseteq X$ and $N\subseteq Y$ be bounded subsets of $X$ and $Y$ respectively. Then $z_{n}=  (x_{n},y_{n}) \in M\oplus_{p}N$ is a $\mathcal{I}$-maximizing sequence in $M\oplus_{p}N$  if and only if $x_{n}$ and $y_{n}$ are $\mathcal{I}$-maximizing sequences in $M$ and $N$ respectively. 
\end{proposition}

\begin{proof}
Firstly, let $\{x_{n}\}_{n\in \mathbb{N}}$ and $\{y_{n}\}_{n\in \mathbb{N}}$ be $\mathcal{I}$-maximizing sequences in $M$ and $N$ respectively. Then there exists $x_0 \in X$ and $y_0 \in Y$ such that $\mathcal{I}$-lim $\Vert x_n - x_0 \Vert$ = $\delta(x_0,M)$ and $\mathcal{I}$-lim $\Vert y_n - y_0 \Vert$ = $\delta(y_0,N).$ Now, from Lemma \ref{I convergence sum product}, 
$$\mathcal{I}-\mbox{lim}\Vert x_n - x_0 \Vert ^p + \Vert y_n - y_0 \Vert^p = \delta(x_0 , M)^p + \delta(y_0 , N)^p$$  
$$\Longrightarrow \mathcal{I}-\mbox{lim} \Vert (x_n,y_n) - (x_0,y_0) \Vert_p ^p = \delta\Big((x_0 , y_0) , M \oplus _p N\Big)^p,$$  
$$\Longrightarrow \mathcal{I}-\mbox{lim} \Vert (x_n,y_n) - (x_0,y_0) \Vert_p  = \delta\Big((x_0 , y_0) , M \oplus _p N\Big).$$ 

Conversely, let $\{(x_{n},y_n)\}_{n\in \mathbb{N}}$ be $\mathcal{I}-$maximizing in $M \oplus_p N.$ Then there exists $(x_0 ,y_0) \in X \oplus_p Y $ such that 
$$\mathcal{I}-\mbox{lim}  \Vert (x_n,y_n) - (x_0,y_0) \Vert_p = \delta((x_0 , y_0) , M \oplus _p N),$$

$$\Longrightarrow \mathcal{I}-\mbox{lim} \Vert x_n - x_0 \Vert ^p + \Vert y_n - y_0 \Vert^{p} = \delta(x_0 , M)^p + \delta(y_0 , N)^p.$$ 
Let $\varepsilon > 0$. Then \\
$\Big\{ n\in \mathbb{N} : \delta(x_0 , M)^p -\Vert x_k - x_0 \Vert^p  \geq \varepsilon\Big\} $ \\ 
 $\subseteq$ $\Big\{ n\in \mathbb{N} : \delta(x_0 , M)^p -\Vert x_k - x_0 \Vert^p + \delta(y_0 , N)^p -\Vert y_k - y_0 \Vert^p \geq \varepsilon\Big\} $  \\
 = $\Big\{ n\in \mathbb{N} : \delta(x_0 , M)^p + \delta(y_0 , N)^p -(\Vert x_k - x_0 \Vert^p + \Vert y_k - y_0 \Vert^p) \geq \varepsilon\Big\}.$  \\
Since,  
$$\Big\{n \in \mathbb{N} : \delta(x_0 , M)^p+ \delta(y_0 , N)^p-\Big(\Vert x_n - x_0 \Vert^p + \Vert y_n - y_0 \Vert^p\Big) \geq \varepsilon\Big\} \in \mathcal{I}$$
$$\Longrightarrow \Big\{ n\in \mathbb{N} : \delta(x_0 , M)^p -\Vert x_k - x_0 \Vert^p  \geq \varepsilon\Big\} \in \mathcal{I}.$$ 
So, $\mathcal{I}$-lim $\Vert x_n - x_0 \Vert^p$ = $\delta(x_0,M)^p.$ Similarly , $\mathcal{I}$-lim $\Vert y_n - y_0 \Vert^p$ = $\delta(y_0,N)^p.$ Lastly, by Lemma \ref{I convergence sum product}, 
$\mathcal{I}$-lim $\Vert x_n - x_0 \Vert$ = $\delta(x_0,M)$ and $\mathcal{I}$-lim $\Vert y_n - y_0 \Vert$ = $\delta(y_0,N).$
\end{proof} 

\begin{corollary} \label{I max}
Suppose $X=\oplus_{l^p} X_i$ and $M_i \subseteq X_i$ for all $i.$ If $\{x_n\}_{n\in \mathbb{N}}$ is $\mathcal{I}$-maximizing in $\oplus_{l^p} M_i$ then the ith co-ordinate sequence $\{x^{(i)}_n\}_{n\in \mathbb{N}}$ is $\mathcal{I}$-maximizing in $M_i~\mbox{for}~1\leq p <\infty$ for all $i.$ 
\end{corollary}

\begin{definition}
A $\subseteq$ $X$ is $\mathcal{I}$-$M$-compact if every $\mathcal{I}$-maximizing sequence in $A$ has a convergent subsequence in $A.$
\end{definition}

\begin{remark} \label{ac}
Clearly, compactness implies $\mathcal{I}$-$M$-compactness. Also, $\mathcal{I}$-$M$-compactness coincides with $M$-compactness for $\mathcal{I}=\mathcal{I}_{Fin}$ where $\mathcal{I}_{Fin}$ is  the class of all finite subsets of $\mathbb{N}.$ 
\end{remark}
	
\begin{proposition} 
If $\mathcal{I}$ is a non-trivial admissible ideal in $\mathbb{N}$ then $M$-compactness and  $\mathcal{I}$-$M$-compactness are equivalent.
\end{proposition}

\begin{proof}
Let $S\subset$ be $M$-compact. Also let $\{x_n\}_{n\in  \mathbb{N}}$ be an $\mathcal{I}$-maximizing sequence in $S.$ Then there exists $x_0 \in X$ such that $ \mathcal{I}$-lim $\Vert x_n - x_0 \Vert = \delta(x_0,S).$ Then by Proposition \ref{ah}, $\{x_{n}\}_{n\in \mathbb{N}}$ has a subsequence $\{x_{n_k}\}_{k\in \mathbb{N}}$ such that $ \Vert x_{n_k} - x_0 \Vert \rightarrow \delta(x_0,S)~\mbox{as}~k\rightarrow \infty.$  Thus $\{x_{n_k}\}_{k\in \mathbb{N}}$ is a maximizing sequence in $S.$ So, by $M$-compactness $\{x_{n_k}\}_{k\in \mathbb{N}}$ has a convergent subsequence and hence $\{x_{n}\}_{n\in \mathbb{N}}$ has a convergent subsequence.

Conversely, let $S$ is $\mathcal{I}$-$M$-compact. Also, let $\{x_{n}\}_{n\in \mathbb{N}}$ be a  maximizing sequence in $S.$ Then there exists $x_0 \in X$ such that $\Vert x_n - x_0 \Vert \rightarrow \delta(x_0,S)~\mbox{as}~n\rightarrow \infty.$ Then by Remark \ref{ag}, $\mathcal{I}$-lim $ \Vert x_n - x_0 \Vert = \delta(x_0,S).$ So, by $\mathcal{I}-M$ compactness, $\{x_{n}\}_{n\in \mathbb{N}}$ has a convergent subsequence $\{x_{n_k}\}_{k\in \mathbb{N}}.$ Hence $S$ is $M$-compact.
\end{proof}


\begin{proposition}
If 	$A$ is $\mathcal{I}$-$M$-compact then $\bar{A}  $ is $\mathcal{I}$-$M$-compact.
\end{proposition}

\begin{proof}
Let $\{x_{n}\}_{n\in \mathbb{N}}$ be a $\mathcal{I}$-maximizing sequence in $\bar{A}.$ Then there exists $x\in X$ such that $\Vert x_n - x \Vert$ is $\mathcal{I}$-convergent to $\delta(x,\bar{A}).$ Let $\varepsilon>0$ .Then there exists sequence $\{y_{n}\}_{n\in \mathbb{N}}$ in $A$ such that $\Vert y_n -x_n \Vert < \frac{\varepsilon}{2}$  $\forall~ n \in \mathbb{N}.$ Now,

$$ \Big| \Vert y_n - x \Vert - \delta(x,A) \Big|\leq \Big| \Vert y_n - x \Vert - \Vert x_n - x \Vert \Big| + \Big| \Vert x_n - x \Vert - \delta(x,\bar{A}) \Big| $$ 

$$\Longrightarrow  \Big| \Vert y_n - x \Vert - \delta(x,A) \Big| < \frac{\varepsilon}{2} + \Big| \Vert x_n - x \Vert - \delta(x,A) \Big|.$$ 

Then 
$$\Big\{ n \in \mathbb{N} :  \Big| \Vert y_n - x \Vert - \delta(x,A) \Big| \geq \varepsilon \Big\}\subseteq \Big\{ n \in \mathbb{N} : \Big| \Vert x_n - x \Vert - \delta(x,A) \Big| \geq \frac{\varepsilon}{2} \Big\}\in \mathcal{I}, $$
$$\Longrightarrow \Big\{ n \in \mathbb{N} :  \Big| \Vert y_n - x \Vert - \delta(x,A) \Big| \geq \varepsilon \Big\}\in \mathcal{I}.$$
	
Thus, $\{y_{n}\}_{n\in \mathbb{N}}$ is an $\mathcal{I}$-maximizing sequence in $A$. Since $A$ is $\mathcal{I}$-$M$-compact so there exists a subsequence $\{y_{n_k}\}_{k\in \mathbb{N}}$ of $\{y_{n}\}_{n\in \mathbb{N}}$ such that $y_{n_k}$ $\rightarrow y_0 \in A$ as $k\rightarrow \infty.$ So, there exists  $k_0 \in \mathbb{N}$ such that $\Vert y_{n_k} - y_0 \Vert < \frac{\varepsilon}{2}$   $\forall~ k \geq k_0.$ Thus, $\forall k \geq k_0,$ 
	
$$\Vert x_{n_k} - y_0 \Vert \leq \Vert x_{n_k} - y_{n_k} \Vert + \Vert y_{n_k} - y_0 \Vert < \frac{\varepsilon}{2} + \frac{\varepsilon}{2}< \varepsilon.$$
	
So, $x_{n_k}\rightarrow y_0$ ($\in A \subseteq \bar{A}$) as $k \rightarrow \infty.$ Hence, $\bar{A}$ is $\mathcal{I}$-$M$-compact.
\end{proof}

\begin{remark}
Converse of the above Proposition is not true since $M$-compactness and $\mathcal{I}-M$ compact coincide for admissible ideals and it is known that $M$-compactness of $\bar{A}$ may  not imply $M$-compactness of $A.$

\end{remark}

We now give an example to show that $\mathcal{I}$-$M$-compact set may not be closed.

\begin{example}
We take $\mathcal{I}$ = $\mathcal{I}_{Fin}$ (class of all finite subsets of $\mathbb{N}$). Then $M$-compact coincides with $\mathcal{I}-M$ compact.
Now since in $\mathbb{R}$,  [-1,0)$\cup$(0,1] is $M$-compact so it is also $\mathcal{I}-M$ compact but it is not closed.
\end{example}

Now in next two different examples, we show that in general, $\mathcal{I}$-$M$-compactness and $M$-compactness are two independent notions. 

\begin{example} \label{ai}
 
Let $A=[-1,1]\setminus \{0\},$ $\mathcal{I}_M=2^M$  where $M = \{ 5,6,7,\dots\}.$

Define, $\{x_{n}\}_{n\in \mathbb{N}}$ in $A$ as,

\begin{equation*}
~x_{n}=
\begin{cases}
1  ~\mbox{if} ~~ n=1,2,3,4;  ~&\\
\frac{1}{n} ~\mbox{if} ~n\geq 5.
\end{cases}
\end{equation*}

Then $\vert x_n \vert$ = 1 $\forall$ $n \in \mathbb{N} \setminus M$ and thus  $\{\vert x_{n}\vert \}_{n\in \mathbb{N}}$ $\mathcal{I}_M-$converges to 1 = $\delta(0,A)$. So, $\{\vert x_{n}\vert \}_{n\in \mathbb{N}}$ is $\mathcal{I}_M-$maximizing sequence in $A$ . But clearly $\{\vert x_{n}\vert \}_{n\in \mathbb{N}}$ has no convergent subsequence in $A$. But, $A$ is $M$-compact. Let $\{x_{n}\}_{n\in \mathbb{N}}$ be a maximizing sequence in $A$. Then there exists $x\in \mathbb{R}$ such that 
$$\vert x_n -x \vert \rightarrow \delta(x,A)= \delta(x,\bar{A}).$$ 
Thus $\{x_{n}\}_{n\in \mathbb{N}}$ is maximizing sequence in $\bar{A}$. Since $\bar{A}$ is compact then it is $M$-compact and hence $\{x_{n}\}_{n\in \mathbb{N}}$ has a convergent subsequence $\{x_{n_k}\}_{k\in \mathbb{N}}.$ So, let $x_{n_k}\rightarrow x_0 \in \bar{A}$ as $k \rightarrow \infty.$ 
Now, $$\vert x_{n_k} -x \vert \rightarrow \vert x_0  -x \vert.$$ 
So, $\vert x_0  -x \vert$ = $\delta(x,\bar{A}).$ Since $A$ is remotal, so there exists $x_1 \in A$ such that $\vert x_1 - x \vert = \delta(x,A)=\delta(x,\bar{A})$. If possible let, $x_0 \notin A.$ Then $x_0, x_1 \in \bar{A}$ such that $x_0 \neq x_1$ and $\vert x_1 - x\vert = \delta(x,\bar{A})= \vert x_0 - x\vert$. Thus, $x=0$ [as the farthest point map, $F(x,\bar{A})$ is singleton $\forall x \neq 0$] and so $\{ x_0, x_1 \}$ = $\{ 1, -1 \}.$ Thus, $x_0$ = 1 or -1, which is a contradiction, since $1, -1$ both in $A.$ 
Thus $x_0 \in A$ and hence $A$ is $M$-compact.
\end{example}

\begin{example} \label{aj}
Let $A=[-1 , 1),$ $\mathcal{I} _P$ = $2^P$  where $P=\{ 2,4,6,\dots\}.$ Here $A$ is not $M$-compact. Indeed, $\{1- \frac{1}{n}\}_{n\in \mathbb{N}}$ is a maximizing sequence for in $A$ but it has no convergent subsequence in $A$.
But $A$ is $\mathcal{I}_P$-$M$-compact. Let $\{x_{n}\}_{n\in \mathbb{N}}$ be a $\mathcal{I}_P$-maximizing sequence in $A$. Then there exists $x\in \mathbb{R}$ such that $\vert x_n - x\vert$ is  $\mathcal{I}_P$-convergent to $\delta(x,A).$
Thus, $\vert x_n - x\vert$ = $\delta(x,A)$  $\forall n \notin P$  $\Longrightarrow$  $\vert x_n - x\vert=\delta(x,A)$  $\forall n \in \{1,3,5,7,\dots\}.$ Claim, $x \geq0.$ If not, let $x<0.$ Then $\delta(x,A)$ =  $\vert x - 1\vert$.  Now, $\vert x_n - x\vert$ $<$ $\vert x - 1\vert$ = $\delta(x,A)$ , which is a contradiction. So, our claim is true.
Thus, $\delta(x,A)$ = $\vert x - (-1)\vert$ = $\vert x + 1\vert.$ Hence, $\vert x_n - x\vert$ = $\vert x + 1\vert$ and $x_n \in A ~\forall~ n \in \{1,3,5,7,\dots\}$ So, $x_n = -1~  \forall~ n \in \{1,3,5,7,\dots\}.$ So, $\{x_{2n+1}\}_{n\in \mathbb{N}}$ is a convergent subsequence of $\{x_{n}\}_{n\in \mathbb{N}}.$
\end{example}

\begin{remark}
Example \ref{aj} also shows that there exists a set, namely, $[-1,1),$ which is not closed as well as not $M$-compact  but is $\mathcal{I}$-$M$-compact. Also, In this case ideal $\mathcal{I}\neq \mathcal{I}_{Fin}.$ 
\end{remark}

\section{ A Variation}
We introduce another form of compactness.

\begin{definition}
$A \subseteq X$ is $\mathcal{I}$-$\mathcal{I}$-$M$-compact if every $\mathcal{I}$-maximizing sequence in $A$ has a $\mathcal{I}$-convergent subsequence in $A.$
\end{definition}

\begin{remark}
From Remark $\ref{ag}$ and Proposition $\ref{ah},$  it is clear that if $\mathcal{I}$ is a non-trivial admissible ideal on $\mathbb{N}$ then $M$-compactness, $\mathcal{I}$-$M$-compactness and $\mathcal{I}$-$\mathcal{I}$-$M$-compactness are all equivalent.
\end{remark}

Next, we give an example to show that compactness does imply $\mathcal{I}$-$\mathcal{I}$-$M$-compact.

\begin{example}
Let $A=[-1,1].$ Clearly $A$ is compact. Let, $\mathcal{I}_M=2^M$ where $M=\{ 5,6,7,\dots\}.$ 
Define $\{x_{n}\}_{n\in \mathbb{N}}$ in $A$ as,
\begin{equation*}
~x_{n}=
\begin{cases}
1   ~\mbox{if} ~~ n=1,2,  ~&\\
-1  ~\mbox{if} ~~ n=3,4,  ~&\\
\frac{1}{n}  ~\mbox{if} ~n\geq 5.
\end{cases}
\end{equation*}
Then $\{\vert x_{n}\vert \}_{n\in \mathbb{N}}$ is $\mathcal{I}_M-$convergent to $1= \delta(0,A).$  and so $\{x_{n}\}_{n\in \mathbb{N}}$ is $\mathcal{I}_M$-maximizing in $A.$ 

Claim: $\{x_{n}\}_{n\in \mathbb{N}}$ has no $\mathcal{I}_M$-convergent subsequence in $A$. If not, let $\{$ $x_{n_k}$ $\}_{k\in \mathbb{N}}$ be an $\mathcal{I}_M$-convergent subsequence of $\{x_{n}\}_{n\in \mathbb{N}}$ which is $\mathcal{I}_M$-convergent to $x.$ Then $x_{n_k}=x~~ \forall~ k \in \mathbb{N} \setminus M \Longrightarrow x_{n_k}=x~~\forall~ k=1,2,3,4,$ a contradiction, since $\{x_{n}\}_{n\in \mathbb{N}}$ cannot have a subsequence of this form.
\end{example}

\begin{remark}
The above example also shows that $\mathcal{I}$-$M$-compact does not imply $\mathcal{I}$-$\mathcal{I}$-$M$-compact. Indeed, $A=[-1,1]$ is also $\mathcal{I}$-$M$-compact.
\end{remark}

\begin{proposition}
If $A$ is $\mathcal{I}$-$\mathcal{I}$-$M$-compact then $\bar{A}$ is $\mathcal{I}$-$\mathcal{I}$-$M$-compact.
\end{proposition}

\begin{proof}
Let $\{x_{n}\}_{n\in \mathbb{N}}$ be a $\mathcal{I}$-maximizing sequence in $\bar{A}.$ Then there exists $x \in X$ such that $\Vert x_n - x \Vert$ is $\mathcal{I}$-convergent to $\delta(x,\bar{A}).$ Let $\varepsilon>0.$ Then there exists sequence $\{y_{n}\}_{n\in \mathbb{N}}$ in $A$ such that $\Vert y_n -x_n \Vert < \frac{\varepsilon}{2}~\forall~ n \in \mathbb{N}.$ Now,
$$ \Big| \Vert y_n - x \Vert - \delta(x,A) \Big|\leq \Big| \Vert y_n - x \Vert - \Vert x_n - x \Vert \Big| + \Big| \Vert x_n - x \Vert - \delta(x,\bar{A}) \Big| $$ 

$$\Longrightarrow  \Big| \Vert y_n - x \Vert - \delta(x,A) \Big| < \frac{\varepsilon}{2} + \Big| \Vert x_n - x \Vert - \delta(x,A) \Big|.$$ 

Then 
$$\Big\{ n \in \mathbb{N} :  \Big| \Vert y_n - x \Vert - \delta(x,A) \Big| \geq \varepsilon \Big\}\subseteq \Big\{ n \in \mathbb{N} : \Big| \Vert x_n - x \Vert - \delta(x,A) \Big| \geq \frac{\varepsilon}{2} \Big\}\in \mathcal{I}, $$

$$\Longrightarrow \Big\{ n \in \mathbb{N} :  \Big| \Vert y_n - x \Vert - \delta(x,A) \Big| \geq \varepsilon \Big\}\in \mathcal{I}.$$

Thus, $\{y_{n}\}_{n\in \mathbb{N}}$ is $\mathcal{I}$-maximizing in $A$. Since $A$ is $\mathcal{I}$-$\mathcal{I}$-$M$-compact, so, there exists a subsequence $\{y_{n_k}\}_{k\in \mathbb{N}}$ of $\{y_{n}\}_{n\in \mathbb{N}}$ such that $\{y_{n_k}\}_{k\in \mathbb{N}}$ is $\mathcal{I}$-convergent to $y_0 \in A.$ 
%
%
Now it can be verified that, $\{x_{n_k}\}_{k\in \mathbb{N}}$ is $\mathcal{I}$-convergent to $y_0 \in A$ $\subseteq \bar{A}$ and hence  $\bar{A}$ is $\mathcal{I}$-$\mathcal{I}$-$M$-compact.
\end{proof}

\begin{remark}
Converse of the above theorem is not be true as it is not true for $M$-compact.
\end{remark}
\begin{Acknowledgement}
The first author's research is funded by the National Board for Higher Mathematics (NBHM), Department of Atomic Energy (DAE), Government of India, Ref No: 0203/11/2019-R$\&$D-II/9249.
The second author's research is funded by the Council of Scientific and Industrial Research (CSIR), Government of India under the Grant Number: $25(0285)/18/\mbox{EMR-II}$.
\end{Acknowledgement}

\end{document}